\documentclass[12pt]{article}

\usepackage{graphicx} 

\title{The double Hall property and cycle covers in bipartite graphs}
\author{
János Barát\thanks{HUN-REN Alfr\'ed R\'enyi Institute of Mathematics, Budapest, Hungary and University of Pannonia, Veszprém, Hungary. The author is supported by National Research, Development and Innovation Office, NKFIH,
K-131529 and ERC Advanced Grant ``GeoScape'' 882971.      E-mail: {\tt barat@renyi.hu}}
	\and 
Andrzej Grzesik\thanks{Faculty of Mathematics and Computer Science, Jagiellonian University, 
Krak\'{o}w, Poland. The author is supported by the National Science Centre grant 2021/42/E/ST1/00193. E-mail: {\tt Andrzej.Grzesik@uj.edu.pl}.}
	\and 
Attila Jung\thanks{ELTE E\"otv\"os Lor\'and University and HUN-REN Alfr\'ed R\'enyi Institute of Mathematics, Budapest, Hungary. The author is supported by the ERC Advanced Grant ``ERMiD'', the NKFIH grant FK132060, and the Thematic Excellence Program TKP2021-NKTA-62 of the National
Research, Development and Innovation Office.
		E-mail: {\tt jungattila@gmail.com}.}
	\and  
 Zolt\'an L\'or\'ant Nagy\thanks{ELTE Linear Hypergraphs  Research Group,
		E\"otv\"os Lor\'and University, Budapest, Hungary. The author is supported by the Hungarian Research Grant (NKFIH) No. PD  134953.  	E-mail: {\tt nagyzoli@cs.elte.hu}.}
  \and
  Dömötör Pálvölgyi\thanks{ELTE E\"otv\"os Lor\'and University and HUN-REN Alfr\'ed R\'enyi Institute of Mathematics, Budapest, Hungary. The author is supported by the ERC Advanced Grant ``ERMiD'' and by the J\'anos Bolyai Research Scholarship of the Hungarian Academy of Sciences, and by the New National Excellence Program \'UNKP-23-5 and by the Thematic Excellence Program TKP2021-NKTA-62 of the National Research, Development and Innovation Office. E-mail: {\tt dom@cs.elte.hu}.}
}
\date{}
\usepackage{a4wide}
\usepackage{amssymb,amsfonts,amsthm,amsmath}
\usepackage[english]{babel}
\usepackage[utf8]{inputenc}
\usepackage{t1enc}
\usepackage{epsfig}
\usepackage{array, comment}
\usepackage{graphicx,epstopdf}
\usepackage{latexsym,color, enumerate}
\usepackage{bbm,url}
\usepackage{setspace}
\usepackage{caption,subcaption}
\usepackage{tikz}

\mathcode`l="8000
\begingroup
\makeatletter
\lccode`\~=`\l
\DeclareMathSymbol{\lsb@l}{\mathalpha}{letters}{`l}
\lowercase{\gdef~{\ifnum\the\mathgroup=\m@ne \ell \else \lsb@l \fi}}%
\endgroup

\newcommand{\PP}{\mathbb{P}}

\newcommand{\C}{\mathcal{C}}

\newtheorem{theorem}{Theorem}[section]

\newtheorem{thm}[theorem]{Theorem}

\newtheorem{prop}[theorem]{Proposition}

\newtheorem{lemma}[theorem]{Lemma}

\newtheorem{remark}[theorem]{Remark}

\newtheorem{problem}[theorem]{Problem}

\newtheorem*{conj*}{Conjecture}

\numberwithin{equation}{section}

\begin{document}

\maketitle

\begin{abstract}
In a graph $G$, the $2$-neighborhood of a vertex set $X$ consists of all vertices of $G$ having at least $2$~neighbors in $X$. 
We say that a bipartite graph $G(A,B)$ satisfies the double Hall property if $|A|\geq2$, and every subset $X \subseteq A$ of size at least $2$ has a $2$-neighborhood of size at least $|X|$. Salia conjectured that any bipartite graph $G(A,B)$ satisfying the double Hall property contains a cycle covering~$A$. 
Here, we prove the existence of a $2$-factor covering $A$ in any bipartite graph $G(A,B)$ satisfying the double Hall property. 
We also show Salia's conjecture for graphs with restricted degrees of vertices in $B$. Additionally, we prove a lower bound on the number of edges in a graph satisfying the double Hall property, and the bound is sharp up to a constant factor.
\end{abstract}

\section{Introduction}

Given a graph $G$, let $V(G)$ denote its vertex set, $E(G)$ denote its edge set.
Let $e(G)$ be the number of edges in $G$, and let $e(X,Y)$ be the number of edges between two disjoint subsets of vertices $X$ and $Y$.
Let $N_G(X)$ denote the neighborhood of a vertex set $X$ in $G$. Likewise, let $N^2_G(X)$ denote the $2$-neighborhood of $X$, which is the set of vertices with at least two neighbors in $X$. 
Whenever the graph is known from the context, we omit the subscript. 

\begin{problem}[Salia~\cite{salia}]\label{main}
	Let $G$ be a bipartite graph with sides $A$ and $B$ satisfying $|N^2(X)|\geq |X|$ for every  $X \subseteq A$ of size at least $2$. 
	For every  $X \subseteq A$, $|X|\geq2$,   there is a cycle $C_{X}$ in $G$ such that $V(C_{X})\cap A=X$.
\end{problem}

If the side $A$ of a bipartite graph $G(A,B)$ satisfies $|A|\geq2$, and $|N^2(X)|\geq |X|$ for every  $X \subseteq A$ of size at least $2$, then $G$ satisfies the \emph{double Hall property}, dHp for short. Thus we call these graphs \emph{dHp graphs}.  Notice that once one shows the existence of a  cycle $C_{X}$ in $G$ such that $V(C_{X})\cap A=X$ for the single set $X=A$ in all dHp graphs $G$, the general statement also follows. Therefore, in order to solve Problem~\ref{main} it is enough to prove that any dHp graph $G(A,B)$ contains a cycle covering~$A$.
We also mention that the double Hall property can be seen as a weaker form of the well-studied concept of robust expansion, see e.g., \cite{Kate}.



Problem~\ref{main} was inspired by the following.

\begin{problem}[Kostochka, Lavrov, Luo and Zirlin~\cite{kostochka2020conditions}]\label{kost}
	Let $G$ be a $2$-connected bipartite graph with sides $A$ and $B$ satisfying $|N^2(X)|\ge |X|$ for every $X \subseteq A$ of size at least $3$. 
	For every $X \subseteq A$, $|X|\geq 3$, there is a cycle $C_{X}$ in $G$ such that $V(C_{X})\cap A=X$.
\end{problem}

Note that Problem~\ref{main} implies Problem~\ref{kost}. 
Problem~\ref{kost} was motivated by the following longstanding conjecture. 
For positive integers $n, m,$ and $\delta$ with $\delta \leq m$, let $G(n,m,\delta)$ denote the set of all bipartite graphs with sides $X$ and $Y$ such that $|X| = n\geq 2, |Y|=m$, and for every $x \in X$, $d(x) \geq \delta$. In 1981, Jackson~\cite{jackson} proved that if  $ \delta\geq \max\{n,\frac{m+2}{2}\}$, then every graph $G\in G(n,m,\delta)$ contains a cycle of length $2n$, i.e., a cycle that covers $X$.
This result is sharp. 
Jackson also conjectured that if  $G \in G(n,m,\delta)$ is 2-connected, then the upper bound on $m$ can be weakened.

\begin{theorem}[Jackson's Conjecture~\cite{jackson,jackson3} solved in~\cite{KLZ}]\label{jacksonconj} Let $m$, $n$ and $\delta$ be positive integers. If $\delta\geq \max\{n, \frac{m+5}{3}\}$,
	then every $2$-connected graph $G \in G(n,m,\delta)$ contains a cycle of length $2n$.\end{theorem} 

The restriction $\delta\geq \frac{m+5}{3}$ cannot be weakened.

\begin{theorem}[Kostochka, Lavrov, Luo and Zirlin~\cite{KLLZ}]\label{jackson6} 
	Let $m$, $n$ and $\delta$ be positive integers. If $\delta\geq \max\{n, \frac{m+10}{4}\}$,
	then every $3$-connected graph $G \in G(n,m,\delta)$ contains a cycle of length $2n$.
\end{theorem}

Theorem~\ref{jackson6} is a natural $3$-connected strengthening of Jackson's Conjecture.

In this paper, we introduce a natural reformulation of Problem~\ref{main} using edge-colorings of complete graphs. Applying earlier results due to Albert, Frieze and Reed \cite{Frieze}, we point out in Theorem \ref{fokszamkorlat} that the bottleneck of the problem is the case in which $B$ contains some vertices of large degree. 
Next, in Theorem \ref{2,n}, we solve the case, when $B$ consists of vertices of very large and very small degree, namely $d(v)\in \{2, |A|\}$ for all $v\in B$. 
In Section~\ref{2-f}, we completely solve a relaxation of Problem~\ref{main} by proving the existence of a covering $2$-factor instead of a single covering cycle. Note that this seems also an equally natural consequence of the double Hall property. 
Section~4 is devoted to extremal results concerning dHp graphs. In Theorems \ref{low} and \ref{up} we determine the minimum number of edges in any dHp graph $G(A,B)$ with $|A|=n$ up to a small constant factor. 
We complete our paper by some related open problems of a geometric flavor.

\section{Problem reformulation -- searching for a  rainbow Hamiltonian cycle} \label{rain}

A \emph{colored multigraph} on $n$ vertices is a complete graph $K_n$ 
such that we assign a set $S_e$ of colors to every edge $e$.
A colored multigraph contains a \emph{rainbow Hamiltonian cycle} if there exists a Hamiltonian cycle and a choice of colors on its edges such that every color appears at most once.

Note that every dHp graph $G(A,B)$ corresponds to a colored multigraph $H$ on $|A|$ vertices in the following way. 
We associate a distinct color $i$ to each vertex $b_i$ of $B$ and assign color $i$ to all edges in the clique induced by $N(b_i)$.
For a vertex set $X \subseteq A$, $|X| \geq 2$, of a dHp graph $G(A,B)$, the set $N_G^2(X) \subseteq B$ corresponds to the set of colors assigned to at least one edge of the clique induced by $X$ in the corresponding colored multigraph $H$.
Since any two vertices of $A$ have a common neighbor in $B$, $H$ is complete.
Moreover, by the double Hall property of $G(A,B)$, for any set of vertices $X\subseteq V(K_n)$ of size at least two, there exist at least $|X|$ colors assigned to at least one edge of the clique induced by $X$.
Additionally, the required cycle in a dHp graph covering all the vertices of~$A$ for $|A|\geq3$ is just a rainbow Hamiltonian cycle in the corresponding colored multigraph. 
In view of this correspondence, Problem~\ref{main} can be stated in the following equivalent form. 

\begin{problem}\label{reform} Let $\C$ be a set of colors.
	We assign a subset of $\C$ to every edge of a complete graph~$K_n$, $n\geq3$, such that 
	\begin{itemize}
		\item for each color $c\in \C$, the edges colored $c$ induce a clique,
		\item for any set of vertices $X\subseteq V(K_n)$ of size at least two, there exist at least $|X|$ colors assigned to at least one edge of the clique induced by $X$. 
	\end{itemize}
	The obtained colored multigraph contains a rainbow Hamiltonian cycle.   
\end{problem}

If none of the color cliques is large, then the existence of a rainbow Hamiltonian cycle follows from a classical result of Albert, Frieze and Reed \cite{Frieze}, solving a conjecture of Hahn and Thomassen~\cite{Hahn}.

\begin{theorem}[Albert, Frieze and Reed \cite{Frieze}]\label{AFR}
	Suppose the edges of a complete graph $K_n$ are colored so that no color
	appears more than $\frac{1}{64}n$ times.\footnote{Note that in the original paper $\frac{1}{32}n$ was claimed, but the calculation had an error, and it in fact gives $\frac{1}{64}n$; see \url{https://www.combinatorics.org/ojs/index.php/eljc/article/view/v2i1r10/comment}.} If $n$ is sufficiently large, then there exists a rainbow Hamiltonian cycle in the graph.
\end{theorem}

In order to apply the theorem above, we use the following two lemmas.
Here $d^+(v)$ and $d^-(v)$ denote the out-degree and in-degree of vertex $v$.

\begin{lemma}[Folklore]\label{atmostone}
	The edges of any multigraph $H$ can be oriented so that 
	$|d^+(v)-d^-(v)|\le 1$ for every vertex $v\in V(H)$.
\end{lemma}

\begin{lemma}\label{ritkitas}
	If $G$ is a colored multigraph corresponding to a dHp graph of maximum degree of vertices in $B$ equal to $\Delta$, then for every edge of $G$ one can choose an assigned color such that each color appears at most $\left\lceil \frac{1}{2}\binom{\Delta}{2}\right\rceil$ times.
\end{lemma}

\begin{proof} Observe that at least two colors are assigned to each edge of $G$. 
	For each edge of $G$, keep exactly $2$ colors and drop all the others.
	We form an auxiliary multigraph $H$ on $|\C|$ vertices by putting an edge between every pair $c$, $c'$ of vertices for each edge of $G$, where exactly colors $c$ and $c'$ were kept, and by orienting the edges of $H$ using Lemma~\ref{atmostone}.
	Now, for each edge $e$ of $G$, we choose the color at the head of the arc corresponding to $e$ in $H$.
	Since the degree of each vertex $c$ of $H$ is at most the size of the clique in color $c$, the statement follows by Lemma~\ref{atmostone}.
\end{proof}

Now we can deduce Problem~\ref{main} 
for dHp graphs of bounded degree in $B$.

\begin{thm}\label{fokszamkorlat}
	Let $G$ be a bipartite graph with sides $A$ and $B$ satisfying $|N^2(X)|\geq |X|$ for every  $X \subseteq A$ of size at least $2$. 
	If $n=|A|$ is large enough, and $d(y)\leq \frac{1}{4}\sqrt{n}$ for any $y\in B$, then there is a cycle in $G$ covering all the vertices of $A$. 
\end{thm}

\begin{proof}
	We apply Theorem \ref{AFR} together with Lemma \ref{ritkitas}, and take into account that the conditions of our theorem yield $\left\lceil \frac{1}{2}\binom{\Delta}{2}\right\rceil< \frac{1}{64}n$. 
\end{proof}

We additionally show that if each color clique is either of full size or of size $2$, then a rainbow Hamiltonian cycle exists.

\begin{theorem}\label{2,n}
	Let $G$ be a bipartite graph with sides $A$ and $B$, $|A|\geq2$, satisfying $|N^2(X)|\geq |X|$ for every  $X \subseteq A$ of size at least $2$. 
	If $d (v) \in \{2, |A|\}$ for each $v \in B$, then there exists a cycle in $G$ covering all the vertices of $A$.
\end{theorem}

To prove this result, we recall the concept of path covers and the theorem of Gallai and Milgram~\cite{GallaiMilgram}.
\begin{theorem}[Gallai and Milgram~\cite{GallaiMilgram}]\label{thm:GallaiMilgram}
	The vertices of any graph $G$ can be partitioned into at most $\alpha(G)$ vertex-disjoint paths.
\end{theorem}

\begin{proof}[Proof of {\bf Theorem \ref{2,n}}]
	Let $G'$ be the colored multigraph corresponding to a dHp graph $G(A,B)$ and assume that $k$ of the color cliques are of full size (call them the large colors) while $|B| - k$ of them are of size $2$ (the small colors). Let $H$ be the graph with vertex set $A$ whose edge set is the union of the small color cliques. By the double Hall property, $\alpha(H) \leq k$ holds.
	Thus, we can partition $A$ into at most $k$ vertex-disjoint paths by Theorem~\ref{thm:GallaiMilgram}. This partition corresponds to isolated vertices and rainbow paths in $G'$ using only small colors. We can connect these $k$ components of $G'$ into a rainbow Hamiltonian cycle by using only one edge from each large color class.
\end{proof}

The conjectured existence of a rainbow Hamiltonian cycle in a colored multigraph corresponding to a dHp graph implies that there is a rainbow path of length $n{-}1$ as well.
This motivated us to determine the longest rainbow path guaranteed by the double Hall property of bipartite graphs $G(A,B)$ with $|A|=n$.
In this direction, we have the following result, which implies the existence of a rainbow path of arbitrary length provided that $n$ is large enough, even under far weaker assumptions on the colored edges.

\begin{prop}
	For non-negative integers $k$ and $l$, there is an integer $n_0(k,l)$ such that if the edges of a graph~$G$ on $n\ge n_0(k,l)$ vertices are colored such that any set of vertices $X$ span at least $|X|-k$ colors among them, then $G$ contains a rainbow path of length $l$.
\end{prop}

\begin{proof}
	We prove the statement by induction on $l$.
	For $l=0$ we can take $n_0(k,0)=1$, while for $l=1$ we can take $n_0(k,1)=k+1$, because considering as $X$ the set of all vertices guarantees that $G$ has at least one edge.
	For larger $l$, we set $n_0(k,l)=(2l-1)!!(k+l)$, that is we take the product of all odd numbers up to $2l-1$, and multiply it by $k+l$.
	Assume that we have proved the statement for $l-1$ and every $k$.
	
	If every vertex is incident to at least $2l$ colors, then by induction we find a rainbow path of length $l-1$ in $G$, and extend it through one of its endpoints with a new color.
	Indeed, the path contains at most $l-1$ colors, so its end is incident to at least $2l-(l-1)=l+1$ colors that differ from all of these.
	Pick $l+1$ such edges from the end.
	At most $l$ of these goes to a vertex of the path, so we still have one that goes to a new vertex, extending the path.
	
	Otherwise, there is a vertex $v$ that is incident to at most $2l-1$ colors.
	By the pigeonhole principle, there is a color, red say such that $v$ is connected to at least $\lceil\frac{n-1}{2l-1}\rceil\ge \frac{n_0(k,l)}{2l-1}= n_0(k+1,l-1)$ vertices by a red edge.
	Denote these neighbors by $S$, and delete the red edges from $G[S]$.
	By the condition, for any $X\subseteq S$ there are at least $|X|-k-1$ colors among them.
	By induction, we can find a path of length $l-1$ in $G[S]$, and extend it with a red edge to $v$ from any of its ends.
\end{proof}

\section{Two-factors} \label{2-f}

A problem closely related to Problem \ref{main}
is to search for a covering $2$-factor instead of a single cycle.
Unlike for the existence of Hamiltonian cycles, there is a necessary and sufficient condition due to Lovász \cite{LovaszL} for the existence of a graph factor for which the degrees have the same parity and are lower and upper bounded.

For a graph $G$, let $g$ and $f$ be non-negative integer-valued functions on $V(G)$ with $g(v) \leq f(v)$ for every $v \in V(G)$. 
A \emph{$(g,f)$-factor} is a spanning subgraph $H$ such that $g(v) \leq d_H(v) \leq f(v)$ for each $v \in V(G)$.
If additionally $g$ and $f$ satisfies $g(v) \equiv f(v)$ (mod 2) for every $v \in V(G)$, then a \emph{$(g,f)$-parity factor} of $G$ is a $(g,f)$-factor such that $d_H(v) \equiv f(v)$ (mod 2) for all $v \in V(G)$. Observe that in a bipartite graph $G(A,B)$, a $2$-factor covering $A$ is a $(g,f)$-parity factor with $f(v) = 2$ for every $v \in V(G)$, $g(v) = 2$ for every $v \in A$ and $g(v) = 0$ for every $v \in B$.
For a set $S \subseteq V(G)$ and a function $\varphi : V(G) \longrightarrow \mathbb{Z}^+$ we denote $\sum_{v \in S} \varphi(v)$ by $\varphi(S)$.

\begin{theorem}[Lovász~\cite{LovaszL}]\label{lovasz}
	Let $G$ be a graph and $f, g : V(G) \longrightarrow \mathbb{Z}^+$ be functions satisfying $g(v) \leq f(v)$ and $g(v) \equiv f(v) \pmod 2$ for every $v \in V(G)$.
	The graph $G$ has a $(g,f)$-parity factor if and only if for every disjoint sets of vertices $S$ and $T$
	\begin{equation}\label{eq:2-factor}
		g(T)+q(S,T)\le 
		f(S)+\sum_{v\in T} d_{G-S}(v),
	\end{equation}
	where $q(S,T)$ denotes the number of components $C$ of $G- (S\cup T)$ for which $g(C)+e(C,T)$ is odd.
\end{theorem}

Using this result we prove the following theorem.

\begin{theorem}\label{2-factor}
	Let $G$ be a bipartite graph with sides $A$ and $B$ satisfying $|N^2(X)|\geq |X|$ for every  $X \subseteq A$ of size at least $2$. 
	For every  $X \subseteq A$, $|X|\geq2$, there is a family $\mathcal{C}_{X}$ of disjoint cycles in~$G$ such that $V(\mathcal{C}_{X})\cap A=X$.
\end{theorem}
\begin{proof}
	It is enough to show that there exists a 2-factor covering $A$ for $|A|\geq2$.
	We apply Theorem~\ref{lovasz} with $f(v) = 2$ for every $v \in V(G)$ and $g(v)=2$ for $v\in A$ and $g(v)=0$ for $v\in B$.
	This way a $(g,f)$-parity factor corresponds to a $2$-factor covering $A$.
	We need to show that inequality $(\ref{eq:2-factor})$ holds for any disjoint subsets $S$ and $T$.
	Notice that $g(C)$ is always even, so this never changes the parity of $g(C)+e(C,T)$, so we can forget about that term.
	Using this, we can rewrite the inequality~$(\ref{eq:2-factor})$ into
	\begin{equation}\label{eq:2-factor-rewritten}
		2|T\cap A|+q(S,T)\le 
		2|S|+\sum_{v\in T} d_{G-S}(v),
	\end{equation}
	where $q(S,T)$ denotes the number of components $C$ of $G- (S\cup T)$ for which $e(C,T)$ is odd.
	
	Firstly, notice that we can assume that $T\cap B=\emptyset$.
	Indeed, if the statement is true for subsets $S$ and $T$, then adding a vertex $v\in B-(S\cup T)$ to $T$ would not change $|T\cap A|$ or $|S|$, while it could increase $q(S,T)$ only by at most $d_{G-S}(v)$.
	
	If $T = \emptyset$, then inequality $(\ref{eq:2-factor-rewritten})$ is trivially satisfied. 
	If $|T|=1$, and $v$ is the only vertex in~$T$, then $q(S,T) \leq d_{G-S}(v)$ as every component $C$ of $G-(S\cup T)$ with odd $e(C,T)$ must contain a neighbor of $v$. 
	Thus, inequality $(\ref{eq:2-factor-rewritten})$ holds unless $S = \emptyset$ and every component of $G-v$ contains exactly one neighbor of $v$. 
	This is not possible, because considering any vertex $w \in A$ different from $v$ and using the assumed condition for $X=\{v,w\}$, we obtain that $v$ must have more than 1 neighbor in the component of $G-v$ containing $w$.
	
	In the remaining case $|T|\geq2$, we apply the assumed condition for $X=T$ and obtain
	\[|T|\le |N^2(T)| \leq |S\cap B|+|N^2(T)-S|.\]
	Now, let $N^{i!}(T)$ be the set of vertices having exactly $i$ neighbors in $T$, and  
	$N^{odd}(T)=\cup_{i \text{ is odd}} N^{i!}(T)$ be the set of vertices that have an odd number of neighbors in $T$.
	Notice that $q(S,T)\le |N^{odd}(T)-S|$ as every component counted 
	in $q(S,T)$ must contain a vertex from $N^{odd}(T)-S$.
	By double counting the edges between $T$ and $B-S$, we have 
	\begin{equation*}
		\begin{aligned}
			\sum_{v\in T} d_{G-S}(v)=\sum_i i|N^{i!}(T)-S|&\ge |N^{odd}(T)-S|+2|N^2(T)-S| \\ &\ge q(S,T)+2|N^2(T)-S|.
		\end{aligned}
	\end{equation*}
	Combining the two inequalities we obtain 
	\begin{equation*}
		\begin{aligned}
			2|T|+q(S,T) &\le (2|S\cap B|+2|N^2(T)-S|) + \left(\sum_{v\in T} d_{G-S}(v) - 2|N^2(T)-S|\right)   \\ 
			&\le 2|S|+\sum_{v\in T} d_{G-S}(v)
		\end{aligned}
	\end{equation*}
	which finishes the proof.    
\end{proof}

For general graphs, not necessarily bipartite, the following result of Belck \cite{Belck} provides a condition for having a 2-factor.

\begin{theorem}[Belck \cite{Belck}]\label{Belck}
	A graph has a $2$-factor if and only if
	\begin{equation}\label{eq:belck}
		|T| \leq |S| + \sum_C \left\lfloor \frac{1}{2}e(C,T)\right\rfloor
	\end{equation}
	holds for each pair of disjoint sets $S$, $T \subseteq V(G)$, where $T$ is an independent set and $C$ ranges over all components of $G - (S\cup T)$.
\end{theorem}

Using this, we can prove an analog of Theorem~\ref{2-factor} for not necessarily bipartite graphs.

\begin{theorem}\label{nonbip-2-factor}
	Let $G$ be an arbitrary graph on at least $2$ vertices satisfying $|N^2(X)|\geq |X|$ for every  $X \subseteq V(G)$ of size at least $2$.  
	Then there exists a family of disjoint cycles covering $V(G)$.
\end{theorem}

\begin{proof}
	We show that $G$ satisfies the condition presented in Theorem~\ref{Belck}.
	Let $T \subseteq V(G)$ be an independent set and $S \subseteq V(G)$ any set disjoint from $T$. 
	If $|T|\leq 1$, then inequality $(\ref{eq:belck})$ easily holds unless $|T|=1$, $S = \emptyset$ and the only vertex $v \in T$ has exactly one neighbor in each component of $G-v$. But this is impossible, because $G$ has at least 2 vertices and using the assumed condition for $X$ containing $v$ and any vertex we get that some component of $G-v$ needs to have more neighbors of $v$.
	In the remaining case $|T|\geq2$, applying the assumption on $G$ for $X=T$, we get $|T|\le |N^2(T)|\le |S|+|N^2(T)-S|$.
	Notice $|N^2(T)-S|\le \sum_C \left\lfloor \frac{1}{2}e(C,T)\right\rfloor$ since every $v \in N^2(T)-S$ is contained in one of the components $C$ of $V(G) - (S \cup T)$ and contributes at least~$2$ to the edge count $e(C,T)$. Thus, 
	\[
	|T|\le |S| + |N^2(T)-S|  \le    
	|S|+\sum_{C} \left\lfloor \frac{1}{2}e(C,T)\right\rfloor\]
	as required.
\end{proof}

\section{Edge density of dHp graphs} \label{edge}

In this section, we determine the order of magnitude of the minimal number of edges in a dHp graph. 

\begin{theorem}\label{low}
	Every dHp graph $G(A,B)$ with $|A|=n$ and without isolated vertices has at least $\frac{1}{2}n\log_2n + |B|$ edges.    
\end{theorem}

\begin{proof}
	The vertices in $B$ of degree $1$ do not influence the double Hall property. 
	These vertices contribute the same amount to the number of edges and to the bound.
	Hence, we may assume that every vertex in~$B$ has degree at least $2$. 
	
	For convenience, instead of the dHp graph $G$, we use the corresponding colored multigraph~$G'$ on $n$ vertices as explained in Section~\ref{rain}. 
	For a colored multigraph $H$, let $m_i(H)$ denote the order of the clique in color $i$ contained in $H$, and $i \in E(H)$ means that color $i$ is assigned to some edge of $H$.  
	For any subset $W$ of $V(G')$, we set 
	$$f(W) = \sum_{i\,:\, i\in E(G'[W])}(m_i(G'[W])-1).$$
	
	The statement is equivalent to showing the inequality $f(V(G')) \geq \frac{1}{2}n\log_2 n$. We proceed by induction on $n$. For $n\in \{2,3\}$ the statement clearly holds, so let $n \geq 4$. 
	For any partition of the vertex set $V(G')=S\dot\cup T$ we have
	$$f(S\dot\cup T)=f(S)+f(T)+|\{i: i\in E(G'[S,T])\}|.$$
	Consider a random partition of $V(G')$ into two parts $S$ and $T$. Observe the expected value of $|\{i: i\in E(G'[S,T])\}|$ is at least $|B|/2$ as every vertex in $B$ has degree at least $2$. 
	Also $|B|\geq n$ by the double Hall property used for $X=A$.
	Together with the inductive hypothesis and the convexity of the function $\frac{1}{2}x\log_2 x$, this yields
	\[f(V(G'))\geq 2\cdot\frac{n}{4}\log_2{\frac{n}{2}}+\frac{n}{2}= \frac{1}{2}n\log_2{n}\]
	as required.
\end{proof}

Next we show the previous bound has the right order of magnitude.

\begin{theorem}\label{up} For every integer $n \geq 2$, there exists a dHp graph $G(A,B)$ with $|A|=n$ and at most $n\log_2n+n$ edges.    
\end{theorem}

\begin{proof}
	We define a class of dHp graphs using binary trees. Let $T$ be an arbitrary binary tree with vertex set $V(T) = A \cup B$, where $A$ is the set of leaves of $T$. Define a bipartite graph $G$ with sides $A$ and $B \cup \{y\}$, where every leaf vertex $a \in A$ is adjacent to its ancestors in $T$ and to vertex~$y$.
	
	We prove that bipartite graphs defined this way are dHp graphs by induction on the number of vertices in $A$. 
	If $|A|=2$, then the double Hall property is satisfied thanks to a common ancestor of the two leaves and vertex~$y$.
	Assume now that $|A| \geq 3$, and let $X \subseteq A$ be any subset of size~$\geq 2$. 
	Choose $u,v \in X$ and $b \in B$ in such a way that $b$ is an ancestor of both $u$ and $v$, but $b$ has no other descendants in $X$. Let $T'$ be a new binary tree obtained by deleting all the descendants of $b$ from $T$ (not just the leaves). The vertex $b$ becomes a leaf in $T'$, so $T'$ has one leaf less than $T$. 
	Let $G'$ be the bipartite graph corresponding to $T'$. Using the inductive hypothesis on $T'$, we have $|N^2_G(X)| \geq |N^2_{G'}(X \cup \{b\} \setminus \{u,v\})| + 1 \geq |X|$.
	
	If $T$ is a complete binary tree with $n$ leaves, then the corresponding bipartite graph $G$ satisfies $e(G) \le n\log_2 n + n$.
\end{proof}

\begin{remark}
	A dHp graph $G(A,B)$ with the minimum number of edges satisfies $|d(u)-d(v)|\leq 2$ for all $u,v\in A$.
\end{remark}

\begin{proof}
	Suppose for contradiction, that there exists $u,v \in A$ with $d(u) > d(v) + 2$. We construct a dHp graph $G'$ with fewer edges and sides $A$ and $B \cup \{b\}$, where $b$ is a new vertex. Let $N_{G'}(v) = N_G(v) \cup \{b\}$, $N_{G'}(u) = N_G(v) \cup \{b\}$ and for all $w \in V(G) \setminus \{u,v\}$ let $N_{G'}(w) = N_G(w)$. We have $e(G') = e(G) + 2 - d(u) + d(v) < e(G)$, which contradicts the minimality of $G$. It remains to show that $G'$ is a dHp graph.
	
	Let $X \subseteq A$ and $|X| \geq 2$. If $u \not\in X$, then $|N^2_{G'}(X)| = |N^2_G(X)| \geq |X|$. If $u \in X$ but $v \not \in X$, then $|N^2_{G'}(X)| = |N^2_G(X \cup \{v\} \setminus \{u\})| \geq |X \cup \{v\} \setminus \{u\}| = |X|$. If $u,v \in X$ and $|X|\geq3$, then $|N^2_{G'}(X)| = |N^2_G(X \setminus \{u\})| + 1 \geq |X|$. Finally, if $X = \{u,v\}$, then $|N^2_{G'}(X)| \geq |N_{G}(v)|+1 \geq 2 = |X|$.
\end{proof}

\section{Concluding remarks and open problems}

We propose some further variants and strengthenings of Problem \ref{main}.
For simplicity, we only state them for paths instead of cycles, as we find them more elegant.
The questions below are non-trivial even if $G$ is a complete bipartite graph.

\begin{problem}
	Let $G$ be a bipartite graph with sides $A$ and $B$, and each edge colored red or blue. 
	For a set $X \subseteq A$ let $N^{RB}(X)$ denote the set of vertices, which are joined to $X$ by a red and by a blue edge as well. 
	Suppose $|N^{RB}(X)|\geq |X|-1$ holds for every  $X \subseteq A$.
	There exists a path with edges in alternating colors covering $A$.
\end{problem}

\begin{problem}
	Let $G$ be a bipartite digraph with sides $A$ and $B$. 
	For a set $X \subseteq A$ let $N^{\leftrightarrow}(X)$ denote the set of vertices, which have an edge to and from $X$. 
	If $|N^{\leftrightarrow}(X)|\geq |X|-1$ holds for every  $X \subseteq A$, then there exists a directed path that covers $A$.
\end{problem}

For all these problems, it is natural to consider geometric variants as well.
For example, what happens in Problem \ref{main} if the vertices of $A$ correspond to a planar point set, the vertices of $B$ to halfplanes, and edges denote containments?

\begin{problem} Given $n$ points and $n$ halfplanes in the plane such that for every $k\ge 2$ for any $k$ points there are at least $k$ halfplanes that contain at least two of them, can we order the points and halfplanes $p_1,H_1,p_2,H_2,\ldots,p_n,H_n$ such that $p_i,p_{i+1}\in H_i$ for $i \in \{1,2,\ldots, n-1\}$ and $p_n, p_1\in H_n$?
\end{problem}

Another possible question is the following.

\begin{problem}
	Consider $n$ points and $n$ (directed) lines in the plane such that for any $k$ points there are at least $k-1$ lines, which intersect the convex hull of the $k$ points, and ask if we can order the points and lines $p_1,l_1,p_2,l_2,\ldots,p_n,l_n$ such that $p_i$ and $p_{i+1}$ are separated by $l_i$?
	Additionally, can we even demand that $p_i$ should be on the left, and $p_{i+1}$ on the right side of the directed line $l_i$?
\end{problem}


\medskip
\textbf{Update.}
Four days after we uploaded our preprint to arXiv, another paper \cite{LV} was uploaded that also studies Problem \ref{main}. The obtained results are quite different, though they also mention a connection to the Gallai--Milgram theorem.

\medskip
\textbf{Acknowledgement.}
The presented work was initiated at the 13th Emléktábla Workshop.
We are grateful to the organizers for the inspiring atmosphere, 
to Nika Salia for proposing the problem and to Panna Gehér and Pál Bärnkopf for their useful comments.
We would also like to thank Kristóf Bérczi for telling us about Theorem \ref{lovasz}.







\end{document}